\documentclass[
final
, nomarks
]{dmtcs-episciences}

\usepackage[utf8]{inputenc}
\usepackage{subfigure}

\usepackage[round]{natbib}

\usepackage{amssymb,amsmath,epsfig,amsthm,tikz,enumerate,hyperref}
\usepackage{multirow}
\usetikzlibrary{calc}

\theoremstyle{plain}
\newtheorem{theorem}{Theorem}[section]

\newtheorem{lemma}[theorem]{Lemma}

\newtheorem{question}[theorem]{Question}

\theoremstyle{definition}
\newtheorem{definition}[theorem]{Definition}

\newcommand\beq{\begin{equation}}
\newcommand\eeq{\end{equation}}
\newcommand\bce{\begin{center}}
\newcommand\ece{\end{center}}
\newcommand\bea{\begin{eqnarray}}
\newcommand\eea{\end{eqnarray}}
\newcommand\bean{\begin{eqnarray*}}
\newcommand\eean{\end{eqnarray*}}
\newcommand\bmt{\begin{multline*}}
\newcommand\emt{\end{multline*}}
\newcommand\ben{\begin{enumerate}}
\newcommand\een{\end{enumerate}}
\newcommand\bit{\begin{itemize}}
\newcommand\eit{\end{itemize}}
\newcommand\brr{\begin{array}}
\newcommand\err{\end{array}}
\newcommand\bt{\begin{tabular}}
\newcommand\et{\end{tabular}}
\newcommand\ba{\begin{array}}
\newcommand\ea{\end{array}}

\newcommand\ms{\medskip}
\newcommand\ul{\underline}

\renewcommand\S{\mathcal S}
\newcommand\D{\mathcal D}
\newcommand\C{\mathcal C}
\newcommand\M{\mathcal M}
\newcommand\IE{\S^{e}}
\newcommand\CIE{\C^{e}}
\newcommand\UC{\mathcal U}
\newcommand\UCE{\UC^{e}}
\DeclareMathOperator\fp{fp}
\DeclareMathOperator\exc{exc}
\DeclareMathOperator\dexc{dexc}
\DeclareMathOperator\cyc{cyc}
\DeclareMathOperator\inv{inv}
\renewcommand\d{\mathbf{d}}
\renewcommand\l{\mathbf{\ell}}
\renewcommand\a{\mathbf{a}}
\newcommand\bij{\Theta}

\newcommand\bijj{\varphi}
\newcommand\Z{\mathbb Z}

\newcommand\open{\begin{tikzpicture}[scale=0.3]
 \draw (0,0)--(1,1); 
 \draw[thick,blue] (.4,1)--(.4,.4)--(1,.4);
\end{tikzpicture}}
\newcommand\close{\begin{tikzpicture}[scale=0.3]
 \draw (0,0)--(1,1);
 \draw[thick,blue] (.6,0)--(.6,.6)--(0,.6);
\end{tikzpicture}}
\newcommand\ubounce{\begin{tikzpicture}[scale=0.3]
 \draw (.1,0)--(1,.9);
 \draw[thick,blue] (0,.45)--(.55,.45)--(.55,1);
\end{tikzpicture}}
\newcommand\lbounce{\begin{tikzpicture}[scale=0.3]
 \draw (0,.1)--(.9,1);
 \draw[thick,blue] (.45,0)--(.45,.55)--(1,.55);
\end{tikzpicture}}
\newcommand\fixed{\begin{tikzpicture}[scale=0.3]
 \draw (0,0)--(1,1);
 \fill(.5,.5) circle (.25);
\end{tikzpicture}}

\def\U{-- ++(1,1) circle(1.5pt)}
\def\D{-- ++(1,-1) circle(1.5pt)}
\def\L{-- ++(1,0) circle(1.5pt)}

\author{Sergi Elizalde\affiliationmark{1}\thanks{Partially supported by NSF grant DMS-1001046, Simons Foundation grant \#280575, and NSA grant H98230-14-1-0125.}}
\title{Continued fractions for permutation statistics}
\affiliation{Department of Mathematics, Dartmouth College}

\keywords{permutation, Motzkin path, continued fraction, cycle diagram, permutation statistic, Bell number}
\received{2017-3-28}
\revised{2017-8-1}
\accepted{2018-6-18}
\begin{document}
\publicationdetails{19}{2018}{2}{11}{3225}
\maketitle
\begin{abstract}
We explore a bijection between permutations and colored Motzkin paths that has been used in different forms by Foata and Zeilberger, Biane, and Corteel. By giving a visual representation of this bijection in terms of so-called cycle diagrams, we find simple translations of some statistics on permutations (and subsets of permutations) into statistics on colored Motzkin paths, which are amenable to the use of continued fractions. We obtain new enumeration formulas for subsets of permutations with respect to fixed points, excedances, double excedances, cycles, and inversions.
In particular, we prove that cyclic permutations whose excedances are increasing are counted by the Bell numbers.
\end{abstract}

\section{Introduction}

The interpretation of continued fractions as generating functions of colored (also referred to as weighted or labeled) Motzkin paths is due to \cite{flajolet_combinatorial_1980}. His celebrated paper also gives several applications of continued fractions to the enumeration of combinatorial objects, including set partitions and permutations.

In the case of permutations, the results in \cite{flajolet_combinatorial_1980} are based on a bijection of \cite{francon_permutations_1979} between permutations and increasing binary trees, which allows Flajolet to obtain continued fractions enumerating permutations with respect to the number of valleys, peaks, double rises, and double falls. 

A second bijection between permutations and colored Motzkin paths was introduced by \cite{biane_permutations_1993}, and it is essentially equivalent to a bijection of \cite{foata_denerts_1990} between permutations and so-called {\it weighted bracketings}. This bijection allows Biane to keep track of the number of inversions. Furthermore, variations of it have been used by \cite{corteel_crossings_2007} to enumerate permutations with respect to the number of weak excedances, crossings and nestings; and by \cite{clarke_new_1997} to prove equidistribution results for several other permutations statistics. 

In this paper we introduce a simple visual interpretation of this second bijection, using what we call the {\em cycle diagram} of the permutation to produce the colored Motzkin path. 
The cycle diagram combines the information contained in the {\em permutation diagram} used by \cite{corteel_crossings_2007} to deal with crossings and nestings (see also \cite{burrill_generating_2016,kasraoui_distribution_2006}), as well as the permutation array, which allows us to keep track of statistics such as inversions and pattern occurrences. The cycle diagram idea was used in~\cite{elizalde_x-class_2011} to enumerate so-called {\it almost increasing} permutations, a generalization of certain permutations studied by \cite{knuth_art_1968} in connection to sorting algorithms. Here we analyze the correspondence between permutations and colored Motzkin paths in order to easily keep track of multiple statistics counting fixed points, cycles, excedances, inversions, and to impose different conditions on the cycles and the excedances of the permutation, as well as pattern-avoidance conditions. 

In Section~\ref{sec:theta} we describe the pictorial correspondence between permutations and colored Motzkin paths and introduce some notation involving cycle diagrams and continued fractions. In Section~\ref{sec:statistics} we find continued fraction expressions for the generating functions of permutations with respect to several statistics, recovering and extending some results in the literature, as well as for occurrences of a monotone consecutive pattern of length $3$. In Section~\ref{sec:subsets} we focus on subsets of permutations satisfying different combinations of conditions such as having a certain cycle structure, avoiding the classical pattern $321$, having increasing excedances, or having unimodal cycles. In particular, Theorem~\ref{thm:bell} proves that cylic permutations with increasing excedances are counted by the Bell numbers. In Section~\ref{sec:pattern} we discuss some known results and open problems regarding the enumeration of pattern-avoiding cyclic permutations. Finally, in Section~\ref{sec:digression} 
we propose a mechanism for interpreting certain combinatorial sequences as 
counting colored Motzkin paths,
which can sometimes turn sequences of positive integers into simpler weight sequences.

\section{Permutations and colored Motzkin paths}\label{sec:theta}

In this section we present the aforementioned bijection between permutations and colored (or weighted) Motzkin paths.

A {\em Motzkin path} of length $n$ is a lattice path from $(0,0)$ to $(n,0)$ with up steps $U=(1,1)$,
down steps $D=(1,-1)$, and level steps $L=(1,0)$, that never goes below the $x$-axis. We define the {\em height} of a step to be
the $y$-coordinate of its  highest point. Let $\M$ be the set of all Motzkin paths. For a path $M\in\M$, let $|M|$ denote its length.
Let $\S_n$ be the symmetric group on $\{1,2,\dots,n\}$.

First we describe a surjective map from permutations in $\S_n$ to Motzkin paths of length $n$.
A permutation $\pi\in\S_n$ can be drawn as an $n\times n$ array with dots in squares $(i,\pi(i))$ for $1\le i\le n$.
Our convention for the coordinates will be as in the cartesian plane, so that square $(i,j)$ is in the $i$th column from the left and the $j$th row from the bottom.
Next, we capture the cycle structure of the permutation on the array by drawing, for each $i$ with $\pi(i)\neq i$, a vertical segment connecting the dot in $(i,\pi(i))$ with the center of the square $(i,i)$,
and a horizontal segment connecting the same dot with the center of the square $(\pi(i),\pi(i))$.
The cycles of $\pi$ are then visualized by simply tracing connected dots, as shown in
Fig.~\ref{fig:cycles}. We call this drawing the {\em cycle diagram} of $\pi$.

\begin{figure}[htb]
\centering
\begin{tikzpicture}[scale=0.5]
 \draw (0,0) grid (12,12);
 \draw (0,0)--(12,12);
 \newcommand{\pp}{5,7,2,4,3,8,1,6,9,12,10,11}
    \foreach \y [count=\x] in \pp {
     \draw[blue,thick] (\x-0.5,\x-0.5)--(\x-0.5,\y-0.5)--(\y-0.5,\y-0.5);
     \draw[fill] (\x-0.5,\y-0.5) circle (5pt);
 }
\end{tikzpicture}
\bigskip

\begin{tikzpicture}[scale=0.5]
 \draw[dotted] (0,-.5)--(0,2.5);
 \draw[dotted] (-.5,0)--(12.5,0);
 \draw[thick] (0,0) circle(1.5pt) \U\U\L\L\D\U\D\D\L\U\L\D;
\end{tikzpicture}
\caption{\label{fig:cycles} The cycle diagram of $\pi=5\,7\,2\,4\,3\,8\,1\,6\,9\,12\,10\,11$, and its associated  Motzkin path $\theta(\pi)$.}
\end{figure}
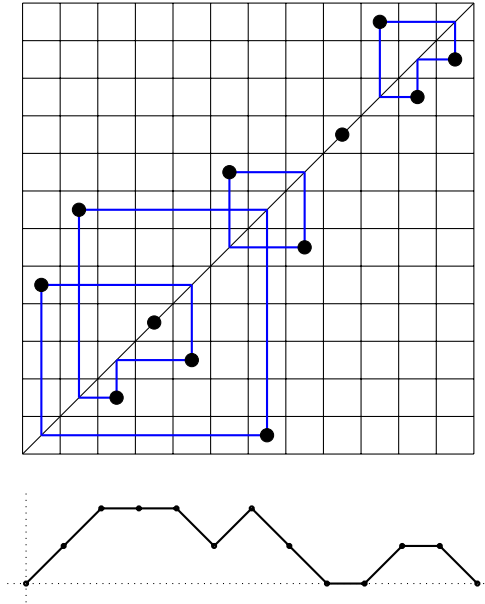

The squares with coordinates $(i,i)$ for some $i$ in
the cycle diagram of $\pi$ can be classified into five types, depending on the location of the dots and segments:
a {\em fixed point} \fixed,
an {\em opening bracket} \open,
a {\em closing bracket} \close,
an {\em upper bounce} \ubounce,
and a {\em lower bounce} \lbounce.
The sequence given by the types of the squares $(i,i)$ for $i$ from $1$ to $n$ is called the {\em diagonal sequence}
of $\pi$, and denoted by $D(\pi)$. Clearly
$D(\pi)\in\{$\fixed,\open,\close,\ubounce,\lbounce$\}^n$. To turn the diagonal sequence into
a Motzkin path of length $n$, we replace each \open\ with a $U$, each \close\ with a $D$, and each \fixed, \ubounce\ and \lbounce\ with an $L$.
Let us denote by $\theta(\pi)$ the resulting Motzkin path. Fig.~\ref{fig:cycles} gives an example of this construction. For each element of the diagonal sequence $D(\pi)$, define its {\em height} to be the height of the corresponding step in the Motzkin path $\theta(\pi)$.

The map $\theta$ is surjective but not one-to-one.
Next we construct, given a Motzkin path $M$ of length $n$, the arrays of all the permutations $\pi$ with $\theta(\pi)=M$. For each $i$ from $1$ to $n$, consider the $i$th step of $M$.

\begin{itemize}
\item If it is a $U$, place a \open\ in square $(i,i)$ of the array, and regard the vertical segment pointing upward and the horizontal segment pointing rightward in \open\ as {\em open} rays. These rays will later be closed by extending them, placing a dot on them, and connecting them with a perpendicular segment, as we will see next.

\item If it is a $D$, place a \close\ in square $(i,i)$ of the array. Letting $h$ be the height of this $D$ step, choose any of the $h$ currently open vertical rays, intersect its extension with the extension of the ray pointing leftward in the newly inserted \close, and place a dot in the intersection. We call this operation {\em closing} an open vertical ray. 
Similarly, choose one of the $h$ currently open horizontal rays, and close it by intersecting it with the ray pointing downward in the newly inserted \close, placing a dot in the intersection. Note that this construction gives $h^2$ choices for which pair of open rays to close.

\item If it is an $L$ at height $h$, chose one of the following $2h+1$ options. One choice is to place a \fixed\ in square $(i,i)$ of the array. Additional $h$ choices
come from placing a \ubounce, choosing one of the $h$ currently open vertical rays, closing it by intersecting it with the leftward pointing ray in the \ubounce\ and placing a dot in
the intersection, and regarding the upward pointing ray in the \ubounce\ as an open ray. The remaining $h$ choices come from placing a \lbounce\ in square $(i,i)$ and proceeding in a symmetric fashion.
\end{itemize}

This process that builds a permutation by placing each diagonal square from left to right while
constructing a cycle diagram with that given diagonal sequence, opening and closing rays accordingly,
will be used repeatedly throughout the paper. We will refer to it as {\em building a cycle diagram from a diagonal sequence}.
See Fig.~\ref{fig:building} for an example of an intermediate step in the construction. Note that the height of a \fixed, \ubounce, \lbounce\ or \close\ in the diagonal sequence equals the number of open horizontal (equivalently, vertical) rays at the time when it is inserted in this process.

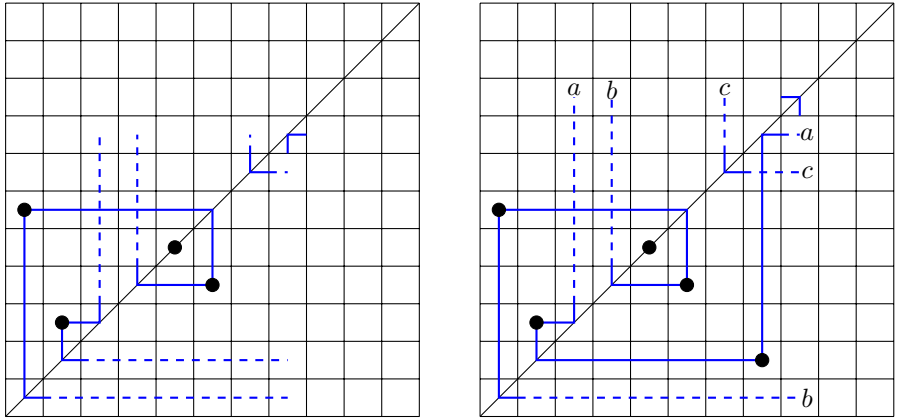
\begin{figure}[htb]
\centering
\begin{tikzpicture}[scale=0.5]
 \draw (0,0) grid (11,11);
 \draw (0,0)--(11,11);
 \draw[blue,thick] (1,0.5)--(0.5,0.5)--(1-0.5,6-0.5)--(6-0.5,6-0.5)--(6-0.5,4-0.5)--(4-0.5,4-0.5)--(4-0.5,4);
 \draw[blue,thick] (2,2-0.5)--(2-0.5,2-0.5)--(2-0.5,3-0.5)--(3-0.5,3-0.5)--(3-0.5,3); 
 \draw[blue,thick] (7,6.5)--(6.5,6.5)--(6.5,7); 
 \draw[blue,thick] (8,7.5)--(7.5,7.5)--(7.5,7); 
 \draw[blue,dashed,thick] (1,0.5)--(7.5,0.5);
 \draw[blue,dashed,thick] (3.5,4)--(3.5,7.5);
 \draw[blue,dashed,thick] (2,1.5)--(7.5,1.5);
 \draw[blue,dashed,thick] (2.5,3)--(2.5,7.5);
 \draw[blue,dashed,thick] (7,6.5)--(7.5,6.5);
 \draw[blue,dashed,thick] (6.5,7)--(6.5,7.5);
 \draw[fill] (1-0.5,6-0.5) circle (5pt);
 \draw[fill] (2-0.5,3-0.5) circle (5pt);
 \draw[fill] (5-0.5,5-0.5) circle (5pt);
 \draw[fill] (6-0.5,4-0.5) circle (5pt);
\end{tikzpicture} \qquad
\begin{tikzpicture}[scale=0.5]
 \draw (0,0) grid (11,11);
 \draw (0,0)--(11,11);
 \draw[blue,thick] (1,0.5)--(0.5,0.5)--(1-0.5,6-0.5)--(6-0.5,6-0.5)--(6-0.5,4-0.5)--(4-0.5,4-0.5)--(4-0.5,4);
 \draw[blue,thick] (8,7.5)--(7.5,7.5)--(8-0.5,2-0.5)--(2-0.5,2-0.5)--(2-0.5,3-0.5)--(3-0.5,3-0.5)--(3-0.5,3); 
 \draw[blue,thick] (7,6.5)--(6.5,6.5)--(6.5,7); 
 \draw[blue,thick] (8,8.5)--(8.5,8.5)--(8.5,8); 
 \draw[blue,dashed,thick] (1,0.5)--(8.5,0.5);
 \draw[blue,dashed,thick] (3.5,4)--(3.5,8.5);
 \draw[blue,dashed,thick] (8,7.5)--(8.5,7.5);
 \draw[blue,dashed,thick] (2.5,3)--(2.5,8.5);
 \draw[blue,dashed,thick] (7,6.5)--(8.5,6.5);
 \draw[blue,dashed,thick] (6.5,7)--(6.5,8.5);
 \draw[fill] (1-0.5,6-0.5) circle (5pt);
 \draw[fill] (2-0.5,3-0.5) circle (5pt);
 \draw[fill] (5-0.5,5-0.5) circle (5pt);
 \draw[fill] (6-0.5,4-0.5) circle (5pt);
 \draw[fill] (8-0.5,2-0.5) circle (5pt);
 \draw (2.5,8.7) node {$a$}; \draw (3.5,8.7) node {$b$}; \draw (6.5,8.7) node {$c$};
 \draw (8.7,0.5) node {$b$}; \draw (8.7,6.5) node {$c$}; \draw (8.7,7.5) node {$a$};
\end{tikzpicture} 
\caption{\label{fig:building} An intermediate step in the process of building a cycle diagram from a diagonal sequence. Open rays, drawn with dashed lines, have been labeled on the right picture to show how each vertical ray is connected to one horizontal ray.}
\end{figure}

It follows from the above construction that
if we assign weight $1$ to each $U$ in the Motzkin path~$M$, weight $h^2$ to each $D$ at height $h$, and weight $2h+1$ to each $L$ at height $h$, then the product of the weights of the steps of $M$ equals the number of permutations $\pi\in\S_n$ with $\theta(\pi)=M$. 
We denote by $w(M)$ the product of the weights of the steps of $M$, and we call it simply the {\em weight of $M$}.
The generating function for weighted Motzkin paths with weight function $w$ is $\sum_{M\in\M} w(M)z^{|M|}$.
By interpreting the weight of a step as the number of possible colors it can receive, the above construction gives a bijection between permutations and colored Motzkin paths, where $D$ steps (resp. $L$ steps) at height $h$ can receive $h^2$ (resp. $2h+1$) colors. We denote this bijection by $\bij$.

The bijection $\bij$ will be used in Section~\ref{sec:statistics} to find the distribution of some permutation statistics, and in Section~\ref{sec:subsets} to enumerate certain subsets of permutations. In both cases, the problem is reduced to counting weighted Motzkin paths, or equivalently, weighted diagonal sequences, where the weights often contain formal variables that keep track of statistics. In the diagonal sequences, we denote the weight of entries at height $h$ as follows:  \close\ has weight $d_h$, \ubounce\ has weight $\ell^a_h$, \lbounce\ has weight $\ell^b_h$, and \fixed\ has weight $\ell^c_h$.
In the associated Motzkin path, a $D$ step at height $h$ has weight $d_h$, and an $L$ step at height $h$ has weight $\ell_h=\ell^a_h+\ell^b_h+\ell^c_h$, with the convention that $\ell^a_0=\ell^b_0=0$, since diagonal sequences have no \ubounce\ or \lbounce\ at height~$0$.

\medskip

We end this section by introducing some notation regarding continued fractions.
Given two sequences $\d=(d_1,d_2,\dots)$ and $\l=(\ell_0,\ell_1,\ell_2,\dots)$, define the Jacobi type continued fraction (J-fraction)
\beq J_{\d,\l}(z)=\dfrac{1}{1-\ell_0z-\dfrac{d_1z^2}{1-\ell_1z-\dfrac{d_2z^2}{1-\ell_2z-\dfrac{d_3z^2}{\ddots}}}}.\label{eq:cf}\eeq
It was shown by \cite{flajolet_combinatorial_1980} that $J_{\d,\l}(z)$ is the generating function for weighted Motzkin paths where $D$ steps (resp. $L$ steps) at height $h$ have weight $d_h$ (resp. $\ell_h$), for each $h\ge0$. Throughout the paper, $U$ steps of Motzkin paths will be assumed to have weight $1$.

We also define the continued fraction
\beq K_{\d,\l}(z)=\ell_0z+\dfrac{d_1z^2}{1-\ell_1z-\dfrac{d_2z^2}{1-\ell_2z-\dfrac{d_3z^2}{\ddots}}}, \label{eq:elevated_cf}\eeq
which is the generating function for {\em elevated} weighted Motzkin paths, meaning that they do not touch the $x$-axis except at the beginning and at the end (where we allow the length-1 path $H$), with weights as above.

\section{Statistics on permutations}
\label{sec:statistics}

An {\em excedance} (resp. fixed point, deficiency) of $\pi\in\S_n$ is a value $\pi(i)$ such that $i<\pi(i)$ (resp. $i=\pi(i)$, $i>\pi(i)$). Let $\exc(\pi)$ (resp. $\fp(\pi)$) denote the number of excedances (resp. fixed points) of $\pi$.
A {\em double excedance} of $\pi$ is a value $\pi(i)$ such that $i<\pi(i)<\pi(\pi(i))$. Let $\dexc(\pi)$ denote the number of double excedances of $\pi$. Let $\cyc(\pi)$ be the number of cycles of $\pi$, and let $\inv(\pi)$ be its number of inversions.

Let $\S=\bigcup_{n\ge0}\S_n$. For a permutation $\pi\in\S_n$, we write $|\pi|=n$ to denote its length. Let $F_{\S}(z)=\sum_{\pi\in\S} z^{|\pi|}=\sum_{n\ge 0} n! z^n$ be the ordinary generating function (OGF) for all permutations.
Recall from Section~\ref{sec:theta} the bijection $\bij$ between $\S_n$ and colored Motzkin paths of length $n$, where $D$ steps (resp. $L$ steps) at height $h$ can receive $h^2$ (resp. $2h+1$) colors.
This bijection yields the well-known continued fraction expansion $F_{\S}(z)=J_{\d,\l}(z)$ with $d_h=h^2$ and $\ell_h=2h+1$. This expansion appears, for example, in \cite[Thm.\ 3B]{flajolet_combinatorial_1980}.

\subsection{Fixed points, excedances, double excedances, cycles, and inversions}\label{sec:fp}

The above enumeration of permutations can be refined by keeping track of statistics that behave well under the bijection $\bij$. Next we consider a few.
\begin{enumerate}[(i)]
\item {\em Fixed points} in the permutation are simply squares of type \fixed, and as such they correspond in the Motzkin path to one of the $2h+1$ color choices for $L$ steps at height $h$, for every $h$. 
\item One can keep track of {\em excedances} by
observing that $i<\pi(i)$ if and only if the square $(\pi(i),\pi(i))$ in the array is of type \close\ or \ubounce.
Thus, excedances correspond in the Motzkin path to $D$ steps (of any of the $h^2$ possible colors), plus $h$ of the $2h+1$ color choices for $L$ steps at height $h$.
\item A {\em double excedance} $i<\pi(i)<\pi(\pi(i))$ in the permutation corresponds to a square $(\pi(i),\pi(i))$ of type~\ubounce, and thus to 
$h$ of the $2h+1$ color choices for $L$ steps at height $h$.
\item In the process of building a cycle diagram from a diagonal sequence described in Section~\ref{sec:theta}, a {\em cycle} is completed every time that a diagonal square of type \close\ closes two rays that belong to the same cycle in the partial diagram constructed so far (see Fig.~\ref{fig:building}). Since each of the open vertical rays is connected to exactly one of the open horizontal rays, it follows that $h$ vertical and $h$ horizontal open rays give rise to $h$ connected pairs before a \close\ at height $h$ is inserted. Thus, out of the $h^2$ color choices for $D$ steps at height $h$, there are exactly $h$ of them that complete a new cycle. Additionally, every fixed point \fixed\ produces a cycle as well.
\item In the process of building a cycle diagram from a diagonal sequence, one can keep track of {\em inversions} by counting how many are forced by each new dot that is placed in the cycle diagram. These can be inversions with another existing dot, or with an open ray that will create an inversion once it is closed. Suppose that at some step of the process there are $h$ open vertical rays and $h$ open horizontal rays when a diagonal square is placed.

If the inserted square is of type \fixed, the dot in this square forces inversions with the dots that will eventually be placed in each of the current open vertical rays (since they will be above and to the left of the recent dot) and open horizontal rays (since they will be below and to the right of the recent dot), contributing $2h$ inversions. 

If the inserted square is of type \ubounce, it closes one of the open vertical rays. Assuming it closes the $j$th open vertical ray from the left, then the dot placed on this ray forces an inversion with each of the $j-1$ open rays to its left, as well as with each of the $h$ open horizontal rays, contributing $j-1+h$ inversions.  A symmetric argument applies if the inserted square is of type \lbounce. Thus, in a generating function where inversions are weighted by $q$, the contribution of a level step at height $h$ in the Motzkin path would be $q^{2h}+2\sum_{j=1}^h q^{j-1+h}=q^{2h}+2q^h[h]_q$, where we use the notation $[h]_q:=1+q+\dots+q^{h-1}$.

If the inserted square is of type \close, suppose it closes the $j$th open vertical ray from the left and the $k$th open horizontal ray from the bottom. The dot placed on the $j$th open vertical ray forces an inversion with each of the $j-1$ open rays to its left, as well as with each of the $h$ open horizontal rays, contributing $j-1+h$ inversions. Similarly, the dot placed on the $k$th open horizontal ray contributes $k-1+h$ inversions, except that the inversion between the two newly placed dots is counted twice. Thus, we have a total of $2h-1+(j-1)+(k-1)$ new inversions. It follows that, weighing each inversion by $q$, the contribution to the generating function of a $D$ step at height $h$ in the Motzkin path is $q^{2h-1}
\left(\sum_{j=1}^h q^{j-1}\right)\left(\sum_{k=1}^h q^{k-1}\right)=q^{2h-1}[h]_q^2$.
\end{enumerate}

Throughout the paper, we will use the variables $x,v,w,t,q$ to mark the statistics number fixed points, number of excedances, number of double excedances, number of cycles and number of inversions, respectively.

If we ignore inversions, it follows from items (i)--(iv) above that the continued fraction of all permutations with respect to the number of fixed points, the number of excedances, the number of double excedances, and the number of cycles is
$$\sum_{\pi\in\S} x^{\fp(\pi)}v^{\exc(\pi)}w^{\dexc(\pi)}t^{\cyc(\pi)}z^{|\pi|}=J_{\d,\l}(z)$$
with $d_h=v(h^2-h+ht)=vh(h-1+t)$ and $\ell_h=xt+h(1+vw)$ for all $h$, that is,
\begin{equation}\label{eq:fp_exc_dexc_cyc}
J_{\d,\l}(z)=\dfrac{1}{1-xtz-\dfrac{vtz^2}{1-(xt+(1+vw))z-\dfrac{2v(1+t)z^2}{1-(xt+2(1+vw))z-\dfrac{3v(2+t)z^2}{\ddots}}}}.
\end{equation}
An equivalent expression was found by \cite[Thm.~3]{zeng_enumerations_1993} (see also~\cite{viennot_theorie_1983}), who derived it by first giving a closed form for the corresponding exponential generating function, and then using an addition formula of Stieltjes and Rogers to deduce the continued fraction.
The special case of \eqref{eq:fp_exc_dexc_cyc} where $v=w=1$ had been obtained in~\cite[Thm.\ 3C]{flajolet_combinatorial_1980} using again a less direct approach, consisting of first applying the bijection by \cite{francon_permutations_1979} from permutations viewed as increasing binary trees to weighted Motzkin paths, and then Foata's fundamental transformation from \cite{foata_theorie_1970}.

While it is difficult to simultaneously keep track of cycles and inversions, we can obtain from items (i)--(iii) and (v) the continued fraction of all permutations with respect to the number of fixed points, the number of excedances, the number of double excedances, and the number of inversions as
$$\sum_{\pi\in\S} x^{\fp(\pi)}v^{\exc(\pi)}w^{\dexc(\pi)}q^{\inv(\pi)}z^{|\pi|}=J_{\d,\l}(z)$$
with $d_h=vq^{2h-1}[h]_q^2$ and $\ell_h=xq^{2h}+(1+vw)q^h[h]_q$ for all $h$, that is,
\begin{equation}\label{eq:fp_exc_dexc_inv}
J_{\d,\l}(z)=\dfrac{1}{1-xz-\dfrac{vqz^2}{1-(xq^2+(1+vw)q)z-\dfrac{vq^3(1+q)^2z^2}{1-(xq^4+(1+vw)q^2(1+q))z-\dfrac{vq^5(1+q+q^2)^2z^2}{\ddots}}}}.
\end{equation}

The special case of \eqref{eq:fp_exc_dexc_inv} where $x=v=w=1$ appears in~\cite[Eq.~(1.1)]{biane_permutations_1993}. 
Generating functions similar to~\eqref{eq:fp_exc_dexc_cyc} and~\eqref{eq:fp_exc_dexc_inv} for the subset of almost-increasing permutations with respect to the above statistics (except for double excedances) appear in~\cite[Thms.\ 6.1, 6.2]{elizalde_x-class_2011}.

Permutations with no double excedances and no double deficiencies (that is, no $i$ with $i>\pi(i)>\pi(\pi(i))$) correspond to cycle diagrams with no \ubounce\ and no \lbounce. Their generating function with respect to $\fp$, $\exc$ and $\cyc$ is
$J_{\d,\l}(z)$ with $d_h=vh(h-1+t)$ and $\ell_h=xt$ for all $h$. If, additionally, we do not allow fixed points (which corresponds to setting $x=0$ in the generating function), then these
are called {\em CUD permutations with all cycles of even length} in~\cite[Prop.\ 2.2]{deutsch_cycle-up-down_2011}, where they are also shown to be counted by the secant numbers.

Continued fractions similar to~\eqref{eq:fp_exc_dexc_cyc} and~\eqref{eq:fp_exc_dexc_inv} have been used by \cite{shin_symmetric_2012,shin_symmetric_2016} to prove $\gamma$-positivity of certain polynomials enumerating permutations with respect to several statistics.

\subsection{Double excedances and consecutive patterns}\label{sec:consec}

Next we show that double excedances in permutations are very closely related to occurrences of the consecutive pattern $\ul{123}$. We underline consecutive patterns to distinguish them from classical patterns, which will appear in Section~\ref{sec:321}.
For $\sigma\in\S_m$, an {\em occurrence} of $\ul{\sigma}$ in $\pi$ {\em as a consecutive pattern} is a subsequence $\pi(i)\pi(i+1)\dots\pi(i+m-1)$ whose elements are in the same relative order as $\sigma(1)\sigma(2)\dots\sigma(m)$. For example, an occurrence of $\ul{123}$ as a consecutive pattern is a subsequence $\pi(i)\pi(i+1)\pi(i+2)$ with $\pi(i)<\pi(i+1)<\pi(i+2)$.

Given $\pi\in\S_n$, consider its cycle decomposition where each cycle is written with its smallest element first, and the cycles are sorted from left to right by decreasing smallest element. Let $\hat\pi$ be the permutation whose one-line notation is obtained by removing the parentheses from this cycle decomposition. The map $\pi\mapsto\hat\pi$ is a bijection from $\S_n$ to $\S_n$, and it is a version of Foata's fundamental transformation from \cite{foata_theorie_1970}.

We claim that the number of double excedances of $\pi$ equals the number of occurrences of $\ul{123}$ in $\hat\pi$ as a consecutive pattern. Indeed,
since cycles are written with their smallest element first, double excedances of $\pi$ correspond to triples of adjacent increasing entries within a cycle. Since there are no increasing subsequences of $\hat\pi$ that straddle two cycles in this cycle decomposition of $\pi$, it follows that double excedances of $\pi$ become precisely occurrences of $\ul{123}$ in $\hat\pi$.

The exponential generating function (EGF) for permutations with respect to the number of occurrences of $\ul{123}$ was found by \cite{elizalde_consecutive_2003}. They showed that, if we let $c_{\ul{123}}(\pi)$ denote the number of occurrences of the consecutive pattern $\ul{123}$ in $\pi$, then
$$\sum_{\pi\in\S} w^{c_{\ul{123}}(\pi)} \frac{z^{|\pi|}}{n!}=\frac{2re^{\frac{1}{2}(1-w+r)z}}{1+w+r-e^{rz}(1+w-r)},$$
where $r=\sqrt{(w-1)(w+3)}$. Using the bijection $\pi\mapsto\hat\pi$ together with Equation~\eqref{eq:fp_exc_dexc_cyc} with $x=v=t=1$, we obtain a continued fraction expression for the corresponding ordinary generating function:
\begin{equation}\label{eq:123}
\sum_{\pi\in\S} w^{c_{\ul{123}}(\pi)} z^{|\pi|}=\dfrac{1}{1-z-\dfrac{z^2}{1-(2+w)z-\dfrac{2^2z^2}{1-(3+2w)z-\dfrac{3^2z^2}{1-(4+3w)z-\dfrac{4^2z^2}{\ddots}}}}}.
\end{equation}
Keeping the variables $v$ and $t$ in Equation~\eqref{eq:fp_exc_dexc_cyc}, we obtain a refinement of Equation~\eqref{eq:123} with respect to the number of ascents (i.e., occurrences of $\ul{12}$) and the number of left-to-right minima, respectively. This is because the bijection $\pi\to\hat\pi$ sends excedances to ascents, and cycles to left-to-right minima.

As a particular case, permutations that avoid the consecutive pattern $\ul{123}$ are in bijection with permutations without double excedances, which are those whose cycle diagram has no diagonal squares of type \ubounce. These correspond, via $\bij$, to Motzkin paths where $L$ steps at height $h$ can have $h+1$ colors and $D$ steps at height $h$ can have $h^2$ colors, or equivalently, to Motzkin paths where all steps ($U$, $L$ and $D$) whose lowest point has ordinate $y$ can have $y+1$ colors. Equation~\eqref{eq:123} for $w=0$ had been conjectured by Paul D. Hanna (personal communication, 2016) based on empirical evidence (see also \cite[A049774]{OEIS}). 

A continued fraction related to Equation~\eqref{eq:123} was given by \cite{flajolet_combinatorial_1980}. Defining a {\em double rise} of $\pi\in\S_n$ to be an occurrence of the consecutive pattern $\ul{123}$ in the sequence $0\pi_1\pi_2\dots\pi_n0$, and shifting the exponent of $z$ up by one, a continued fraction for permutations with respect to the number of double rises appears in~\cite[Thm.\ 3A]{flajolet_combinatorial_1980}.

\section{Subsets of $\S_n$}
\label{sec:subsets}

In this section we enumerate subsets of permutations by restricting the choices of open rays that can be closed by squares of type \close, \ubounce\ and \lbounce\ in the cycle diagram. This is equivalent to restricting the color choices of the $D$ and $L$ steps of the Motzkin path.

For each subset of permutations, we express the generating function with respect to several statistics as a continued fraction. Tab.~\ref{tab:summary} summarizes the continued fractions obtained in this section. For simplicity, the table only shows the univariate generating function, and not the refinements with respect to statistics.

\begin{table}
{\footnotesize \begin{tabular}{|c||c|c|c|c|c|c|}
  \hline
  {\bf \bt{c} Subset of \\ permutations\et } & \bt{c} $d_h$\\ \close \et & \bt{c} $\ell^a_h$ \\ \ubounce \et & \bt{c} $\ell^b_h$ \\
  \lbounce \et & \bt{c} $\ell^c_h$\\ \fixed \et & \bt{r} $\ell_h=\ell^a_h+$ \\ $\ell^b_h+\ell^c_h$ \et  & Counting formula \\ \hline \hline

  {\bf all}  & $h^2$ & $h$ & $h$ & $1$ & $2h+1$ & $n!$ \\ \hline

  {\bf cyclic$^*$}  & \bt{c} $h(h-1)$ \\ (n) \\ $(d_1=1)$ \et &  $h$ & $h$ & $0$ & $2h$ & $(n-1)!$ \\ \hline

  {\bf \bt{c} 321-avoiding \\ ($=$ nonnesting) \et}  & $1$ (o) & $1$ (o) & $1$ (o) & \multirow{2}{*}{$\ba{c} 0 \\ (\ell^c_0=1) \ea$} & \multirow{3}{*}{$\ba{c} 2\\ (\ell_0=1) \ea$} & \multirow{3}{*}{$C_n$} \\ \cline{1-4}
  
  {\bf \bt{c} with unimodal \\ noncrossing cycles \\ and no nested \\ fixed points \et} &  $1$ (i) & $1$ (i) & $1$ (i) & & & \\ \cline{1-5}
  
  {\bf noncrossing}  & $1$ (i) & $0$ & $1$ (i) & $1$ & & \\ \hline

  {\bf \bt{c} with increasing\\ excedances \et}  &  \bt{c} $h$\\ (o/any)\et & $1$ (o) & $h$ & $1$ & $\ba{c} h+2 \\ (\ell_0=1) \ea$ & \cite[A074664]{OEIS} \\ \hline

  {\bf \bt{c} with increasing\\ weak excedances \et}  &  \bt{c} $h$\\ (o/any)\et & $1$ (o) & $h$ & $\ba{c} 0 \\ (\ell^c_0=1) \ea$ & $h+1$ & $B_n$ \\ \hline

  {\bf \bt{c} cyclic$^*$ with \\ increasing\\ excedances \et}  & \bt{c} $h-1$ \\ (o/n) \\ $(d_1=1)$ \et & $1$ (o) & $h$ & $0$ & $h+1$ & $B_{n-1}$\\ \hline

  {\bf \bt{c} with unimodal\\ cycles\et}  & $h$ (m) & $h$ & $h$ & $1$ & $2h+1$ & \bt{l} EGF is \\ $\exp\left(\dfrac{e^{2z}{+}2z{-}1}{4}\right)$ 
  \et\\ \hline
  
  {\bf  \bt{c} with unimodal \\ cycles and \\ increasing\\ excedances \et} & \bt{c} $1$ (o/m) \et & $1$ (o) & $h$  & $1$ & $\ba{c} h+2 \\ (\ell_0=1) \ea$ & \bt{c} (closed form \\ unknown) \et \\ \hline

  {\bf \bt{c} with increasing \\  excedances \\ and increasing \\ deficiencies \et}  & $1$ (o) & $1$ (o) & $1$ (o) & \multirow{2}{*}{$1$} & \multirow{2}{*}{$\ba{c} 3 \\ (\ell_0=1) \ea$} & \multirow{2}{*}{\bt{l} OGF is \\ $\dfrac{2}{1{+}x{+}\sqrt{1{-}6x{+}5x^2}}$ \et} \\ \cline{1-4}
  {\bf \bt{c} with \\ unimodal \\ noncrossing \\ cycles \et} & $1$ (i) & $1$ (i) & $1$ (i) &&& \\
  \hline

 {\bf \bt{c} with no double \\ excedances \\ or double \\ deficiencies \et} & $h^2$ & $0$ & $0$  & $1$ & $1$ & EGF is $\dfrac{e^z}{\cos z}$\\ \hline

 {\bf involutions} & \bt{c} $h$ (m) \et & $0$ & $0$  & $1$ & $1$ &  EGF is $e^{z+z^2/2}$\\ \hline
 
  {\bf \bt{c} $321$-avoiding \\ involutions\et } & \bt{c} $1$ (m) \et & $0$ & $0$  & $\ba{c} 0 \\ (\ell^c_0=1) \ea$ & $\ba{c} 0 \\ (\ell_0=1) \ea$ &  $\displaystyle\binom{n}{\lfloor n/2 \rfloor}$\\ \hline
\end{tabular}
}
\caption{Summary of cycle diagram restrictions and resulting continued fractions for the subsets of permutations considered in Section~\ref{sec:subsets}. The restrictions on rays to be closed
are indicated by the following abbreviations: innermost (i); outermost (o); matching (m), i.e., requiring the vertical and horizontal rays to be connected; nonmatching (n).
For subsets of cyclic permutations (marked with a~*), the resulting Motzkin paths are elevated and the continued fraction is given by Equation~\eqref{eq:elevated_cf}. For all the other subsets, the continued fraction is given by Equation~\eqref{eq:cf}.}
\label{tab:summary}
\end{table}

\subsection{Cyclic permutations}\label{sec:cyclic}

The same idea that we used in Section~\ref{sec:fp}(iv) to keep track of cycles in permutations allows us to restrict our enumeration to permutations that consist of one cycle, which we call cyclic permutations.
Denote by $\C_n$ the set cyclic permutations in $\S_n$, and let $\C=\bigcup_{n\ge0}\C_n$. Recall that $|\C_n|=(n-1)!$.

In the expansion of~\eqref{eq:fp_exc_dexc_cyc}, cyclic permutations correspond to the terms where the exponent of $t$ is~$1$. Thus,
the generating function for cyclic permutations with respect to the statistics $\fp$, $\exc$ and $\dexc$ can be obtained from~\eqref{eq:fp_exc_dexc_cyc} by subtracting $1$, dividing by $t$, and then setting $t=0$. However, we will see that it is also possible to directly obtain a continued fraction of the form~\eqref{eq:elevated_cf}.

Consider again the process that builds a cycle diagram from a diagonal sequence. In order to build a cyclic permutation, we have to restrict the possible vertical and horizontal rays that can be closed when a \close\ is placed,  so that the rays that we close are not connected to each other before this diagonal square is placed. 
As discussed in item~(iv) from Section~\ref{sec:fp} and in Fig.~\ref{fig:building}, each open vertical ray is connected to exactly one open horizontal ray. Thus, for each square of type \close\ (equivalently, each $D$ step in the Motzkin path) at height $h$, there are $h(h-1)$ possible pairs of non-connected rays to close, with the exception of the rightmost \close\ (equivalently, the last $D$ step of the path), for which there is one choice, namely to close the unique open horizontal and vertical rays, creating the only cycle in the permutation. Additionally, for it to be cyclic, a permutation cannot have any fixed points \fixed\ unless it has length~$1$. Using items (ii) and (iii) as well, it follows that
$\sum_{\pi\in\C} x^{\fp(\pi)}v^{\exc(\pi)}w^{\dexc(\pi)}z^{|\pi|}=K_{\d,\l}(z)$ with $d_1=v$, $\ell_0=x$, $d_h=h(h-1)v$ for $h\ge2$, and $\ell_h=h(1+vw)$ for $h\ge1$, that is, the desired continued fraction is
$$K_{\d,\l}(z)=xz+\dfrac{vz^2}{1-(1+vw)z-\dfrac{2vz^2}{1-2(1+vw)z-\dfrac{6vz^2}{1-3(1+vw)z-\dfrac{12vz^2}{\ddots}}}}.$$

\subsection{$321$-avoiding permutations}\label{sec:321}

In this section we use the classical definition of permutation patterns, not to be confused with consecutive patterns as defined in Section~\ref{sec:consec}. For $\sigma\in\S_m$, a permutation $\pi\in\S_n$ is said to avoid $\sigma$ if there is no subsequence $\pi(i_1)\pi(i_2)\dots\pi(i_m)$ with $i_1<\dots<i_m$ whose elements are in the same relative order as $\sigma(1)\sigma(2)\dots\sigma(m)$. For example, an occurrence of $321$ is a decreasing subsequence of length~$3$.

Let $\S_n(321)$ denote the set of $321$-avoiding permutations in $\S_n$, and let $\S(321)=\bigcup_{n\ge0}\S_n(321)$. It is well known (see \cite{knuth_art_1968}) that $|\S_n(321)|=C_n$, the $n$th Catalan number.
The following fact about $321$-avoiding permutations is often used in the literature. A non-excedance refers to a value $\pi(i)$ such that $i\ge\pi(i)$, that is, a fixed point or deficiency.

\begin{lemma}\label{lem:321}
A permutation is $321$-avoiding if and only if both its excedances and its non-excedances form increasing subsequences.
\end{lemma}

\begin{proof}
Clearly, a merge of two increasing subsequences cannot contain an occurrence of $321$, proving the `if' direction. To prove the converse, suppose that $\pi$ contains a pair of decreasing non-excedances, say $\pi(k)<\pi(j)\le j<k$. We will show that there is some $i<j$ such that $\pi(i)>\pi(j)$, and thus $\pi(i)\pi(j)\pi(k)$ is an occurrence of $321$ in $\pi$.  Indeed, if the entries $\pi(1),\pi(2),\dots,\pi(j-1)$ were all less than $\pi(j)$, then, using that $\pi(k)<\pi(j)$, there would be $j$ entries taking no more than $\pi(j)-1\le j-1$ different  values, which is a contradiction.
\end{proof}

Again, consider the process that builds the cycle diagram of a permutation from a diagonal sequence. 
Requiring excedances to form an increasing subsequence is equivalent to requiring that every vertical ray that we close is the leftmost open ray at that time. This prevents us from creating a pair of decreasing excedances by first closing a vertical ray and later closing another vertical ray to the left of it. Similarly, having the sequence of deficiencies to be increasing is equivalent to requiring every horizontal ray that we close to be the bottommost open ray at that time. We use the term {\em outermost} ray to refer to the leftmost open vertical ray or the bottommost open horizontal ray.
Additionally, in order for the sequence of non-excedances to be increasing, fixed points \fixed\ can only occur when there are no open rays.

In summary, a permutation has increasing excedances and non-excedances (equivalently, avoids $321$ by Lemma~\ref{lem:321})
if and only if its cycle diagram is obtained from a diagonal sequence with no fixed points at height $\ge1$
by always closing the outermost open rays every time that a \close, \ubounce\ or \lbounce\ is encountered.

By the above argument, $\bij$ restricts to a bijection between $321$-avoiding permutations and colored Motzkin paths where $D$ steps can only receive one color (since for each \close, the rays to close are forced), $L$ steps at height $h\ge1$ can receive two colors (corresponding to inserting a \ubounce\ or a \lbounce), and $L$ steps at height $0$ can receive one color (corresponding to inserting a \fixed).
These bicolored Motzkin paths are well-known to be counted by the Catalan numbers (see \cite{delest_algebraic_1984}). Keeping track of fixed points, excedances, double excedances and inversions and using items (i)--(iii) and (v) from Section~\ref{sec:fp}, their continued fraction expansion is
$\sum_{\pi\in\S(321)} x^{\fp(\pi)}v^{\exc(\pi)}w^{\dexc(\pi)}q^{\inv(\pi)}z^{|\pi|}=J_{\d,\l}(z)$ where $\ell_0=x$, $d_h=vq^{2h-1}$ and $\ell_h=(1+vw)q^h$ for $h\ge1$, that is,
$$J_{\d,\l}(z)=\dfrac{1}{1-xz-\dfrac{vqz^2}{1-(1+vw)qz-\dfrac{vq^3z^2}{1-(1+vw)q^2z-\dfrac{vq^5z^2}{\ddots}}}}.$$

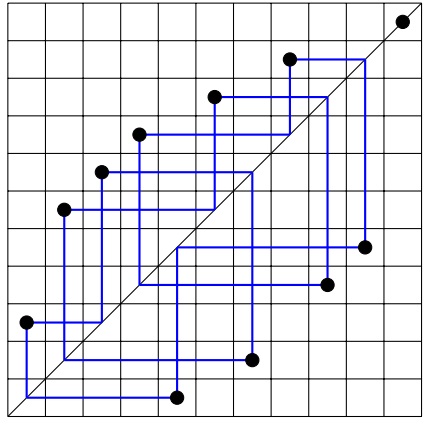
\begin{figure}[htb]
\centering
\begin{tikzpicture}[scale=0.5]
 \draw (0,0) grid (11,11);
 \draw (0,0)--(11,11);
 \newcommand{\pp}{3,6,7,8,1,9,2,10,4,5,11}
    \foreach \y [count=\x] in \pp {
     \draw[blue,thick] (\x-0.5,\x-0.5)--(\x-0.5,\y-0.5)--(\y-0.5,\y-0.5);
     \draw[fill] (\x-0.5,\y-0.5) circle (5pt);
 }
\end{tikzpicture}
\caption{\label{fig:321} The $321$-avoiding permutation $\pi=3\,6\,7\,8\,1\,9\,2\,10\,4\,5\,11$.}
\end{figure}

A slight variation of this continued fraction without the statistic $\dexc$ appears in~\cite[Thm.\ 7.3]{cheng_inversion_2013}, where it is obtained using a bijection similar to the above restriction of $\bij$, although without the visual description. A remarkable property of this bijection is that the statistic $\inv$ on a $321$-avoiding permutation becomes the area under the corresponding Motzkin path (and above the $x$-axis). This property does not hold for the general definition of $\Theta$ on arbitrary permutations.

A closed form for the generating function of $321$-avoiding permutations with respect to the statistics $\fp$ and $\exc$ appears in~\cite{elizalde_fixed_2012}:
$$\sum_{\pi\in\S(321)} x^{\fp(\pi)}v^{\exc(\pi)}z^{|\pi|}=\frac{2}{1+(1-2x+v)z+\sqrt{1-2(1+v)z+(1-v)^2z^2}}.$$

\subsection{Other Catalan classes}\label{sec:catalan}
The construction of cyclic diagrams of $321$-avoiding permutations given in Section~\ref{sec:321} can be modified to obtain other subsets of permutations counted also by the Catalan numbers. For example, the restriction of allowing only the outermost open (vertical and horizontal) ray to be closed at any time can be replaced by allowing  only the innermost open (vertical and horizontal) ray to be closed, without changing the condition that \fixed\ can only occur at height $0$. This new restriction of $\bij$ gives a bijection between bicolored Motzkin paths as before, and permutations with {\em unimodal noncrossing cycles} and {\em no nested fixed points}, defined as follows. 
\begin{definition}\label{def:unimodal}
Let $\pi\in\S_n$. We say that
\begin{enumerate}[(a)] 
\item $\pi$ has {\em noncrossing cycles} if the partition of $\{1,2,\dots,n\}$ induced by the cycles of $\pi$ is noncrossing, that is, there are no $i<j<k<l$ such that $i,k$ belong to one cycle and $j,l$ belong to another;
\item a cycle of $\pi$ is {\em unimodal} if, when written with its smallest element first as $(a_1,a_2,\dots,a_k)$, there exists some $1\le i\le k$ such that $a_1<\dots<a_i>\dots>a_k$;
\item a fixed point $j$ is {\em nested} there exist $i,k$ with $i<j<k$ such that $\pi(i)=k$ or $\pi(k)=i$.
\end{enumerate}
\end{definition}
Fig.~\ref{fig:unimodal_noncrossing} gives an example of a permutation with unimodal noncrossing cycles and no nested fixed points. By construction, the number of such permutations is the Catalan number $C_n$, and their refined enumeration with respect to fixed points, excedances, and double excedances coincides with that of $\S_n(321)$. Keeping track of the number of cycles and inversions as well, with the usual variables, the generating function for 
permutations with unimodal noncrossing cycles and no nested fixed points is $J_{\d,\l}(z)$ with $\ell_0=xt$, $d_h=vtq^{4h-3}$ and $\ell_h=(1+vw)q^{2h-1}$ for $h\ge1$.
 
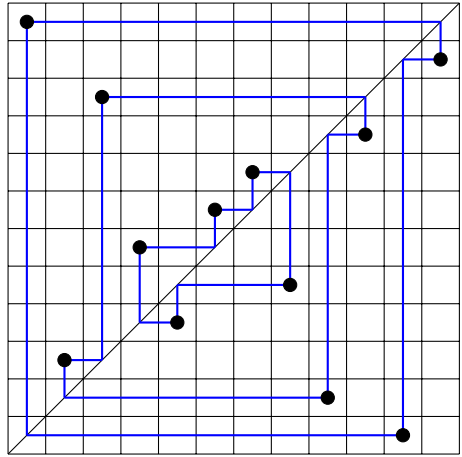
\begin{figure}[htb]
\centering
\begin{tikzpicture}[scale=0.5]
 \draw (0,0) grid (12,12);
 \draw (0,0)--(12,12);
 \newcommand{\pp}{12,3,10,6,4,7,8,5,2,9,1,11}
    \foreach \y [count=\x] in \pp {
     \draw[blue,thick] (\x-0.5,\x-0.5)--(\x-0.5,\y-0.5)--(\y-0.5,\y-0.5);
     \draw[fill] (\x-0.5,\y-0.5) circle (5pt);
 }
\end{tikzpicture}
\caption{\label{fig:unimodal_noncrossing} The permutation $\pi=12\,3\,10\,6\,4\,7\,8\,5\,2\,9\,1\,11=(1,12,11)(2,3,10,9)(4,6,7,8,5)$
has unimodal noncrossing cycles and no nested fixed points.}
\end{figure}

Another set counted by the Catalan numbers consists of permutations obtained from diagonal sequences by closing always the outermost open vertical ray and the innermost open horizontal ray, and not allowing \fixed\ except at height $0$.

Yet another variation ensues by closing always the innermost open ray, but not allowing any~\ubounce, and instead allowing \fixed\ at any height. The resulting permutations are precisely {\em noncrossing} permutations in the sense of \cite{corteel_crossings_2007}, 
whereas $321$-avoiding permutations coincide with {\em nonnesting} permutations in this setting. More generally, the distribution of crossings and nestings in permutations can be obtained using the same ideas, as has been done by \cite{corteel_crossings_2007}.

\subsection{Permutations with increasing excedances with respect to cycles}\label{sec:bell}
In this section we show that cyclic permutations whose subsequence of excedances is increasing are counted by the Bell numbers. As discussed in Section~\ref{sec:321}, a permutation has increasing excedances if and only if in the process that builds its cycle diagram from a diagonal sequence, the leftmost open vertical ray is closed each time that a \ubounce\ or a \close\ is encountered.
At this time, this leftmost open vertical ray is connected to one of the open horizontal rays. Thus, when a square of type \close\ is placed in the diagonal and the leftmost open vertical is closed, any of the open horizontal rays can be closed, but there is precisely one of them whose closure creates a new cycle in the permutation.

Letting $\IE\subset\S$ be the subset of permutations with increasing excedances, it follows that
$$\sum_{\pi\in\IE} x^{\fp(\pi)}v^{\exc(\pi)}w^{\dexc(\pi)}t^{\cyc(\pi)}z^{|\pi|}=J_{\d,\l}(z)$$
with $d_h=v(h-1+t)$ for all $h$, $\ell_0=xt$ and $\ell_h=xt+vw+h$ for $h\ge1$, that is, the desired continued fraction is
\begin{equation}\label{eq:Iexc}
J_{\d,\l}(z)=\dfrac{1}{1-xtz-\dfrac{vtz^2}{1-(1+xt+vw)z-\dfrac{v(1+t)z^2}{1-(1+xt+vw)z-\dfrac{v(2+t)z^2}{\ddots}}}}.
\end{equation}
Setting $x=v=w=t=1$, we obtain sequence \cite[A074664]{OEIS}.

Now let $\CIE\subset\C$  be the subset of cyclic permutations with increasing excedances. The same argument gives
$$\sum_{\pi\in\CIE} x^{\fp(\pi)}v^{\exc(\pi)}w^{\dexc(\pi)}t^{\cyc(\pi)}z^{|\pi|}=K_{\d,\l}(z)$$ 
with $d_1=v$, $\ell_0=x$, $d_h=(h-1)v$ for $h\ge2$, and $\ell_h=vw+h$ for all $h\ge1$.
In particular, setting $x=v=w=1$ and letting $\CIE_n=\CIE\cap\S_n$, we have
$\sum_{n\ge1}|\CIE_n|\,z^n=K_{\d,\l}(z)$ with $d_1=1$, $d_h=h-1$ for $h\ge2$, and $\ell_h=h+1$ for all $h$, that is,
\beq\label{eq:FB} \sum_{n\ge1}|\CIE_n|\,z^n=z+\dfrac{z^2}{1-2z-\dfrac{z^2}{1-3z-\dfrac{2z^2}{1-4z-\dfrac{3z^2}{\ddots}}}}.\eeq
Our next result states that the coefficients of this generating function are the Bell numbers. Let $B_n$ denote the $n$th Bell number, which is the number of partitions of an $n$-element set.

\begin{theorem}\label{thm:bell}
For $n\ge1$, $$|\CIE_n|=B_{n-1}.$$
\end{theorem}

We will give a bijective proof of Theorem~\ref{thm:bell} using the following result of \cite{flajolet_combinatorial_1980}.

\begin{lemma}[{\cite[Prop.\ 8]{flajolet_combinatorial_1980}}]\label{lem:bell}
There is an explicit bijection between set partitions of $\{1,2,\dots,n\}$ and colored (or weighted) Motzkin paths of length $n$ where $\ell_h=h+1$ and $d_h=h$. In particular, 
$$\sum_{n\ge0}B_n z^n=\dfrac{1}{1-z-\dfrac{z^2}{1-2z-\dfrac{2z^2}{1-3 z-\dfrac{3z^2}{\ddots}}}}.$$
\end{lemma}

\begin{proof}[Proof of Theorem~\ref{thm:bell}]
The restriction of $\bij$ to $\CIE_n$ used to obtain Equation~\eqref{eq:FB} gives a bijection between $\CIE_n$ and the set of elevated colored Motzkin paths of length $n$ where each $L$ step at height $h$ receives some color $r$ with $0\le r\le h$, and each $D$ step at height $h\ge2$ receives some color $r$ with $0\le r\le h-2$ (down steps at height $h=1$ receive color $0$). Let $\mathcal{E}_n$ be this set of colored paths.

By Lemma~\ref{lem:bell}, set partitions of $\{1,2,\dots,n-1\}$ (which are counted by $B_{n-1}$) are in bijection with colored Motzkin paths of length $n-1$ where each $L$ step at height $h$ receives some color $r$ with $0\le r\le h$, and each $D$ step at height $h$ receives some color $r$ with $0\le r\le h-1$. Let $\mathcal{B}_{n-1}$ be this set of colored paths.

Our goal is to construct a bijection $\bijj$ between $\mathcal{E}_n$ and $\mathcal{B}_{n-1}$. Given a  path $M\in\mathcal{E}_n$, consider two cases: 
\begin{enumerate}
\item If no $L$ step of $M$ receives a color equal to its height, then write $M=UPD$, where $P$ is a colored Motzkin path. Let $\bijj(M)=LP$, where the colors of the steps of $P$ are preserved, and the new $L$ step receives color $0$.
\item Otherwise, write $M=P_1L P_2D$, where the step between $P_1$ and $P_2$ is the rightmost $L$ step of $M$ whose color equals its height, say $h$. Let $\bijj(M)=P_1DP_2$, where the new $D$ step receives color $h-1$, and the colors of the steps of $P_1$ and $P_2$ are preserved.
\end{enumerate}

In both cases, we have $\bijj(M)\in\mathcal{B}_{n-1}$ by construction. An example is given in Fig.~\ref{fig:bijj}. 

\begin{figure}[htb]
  \begin{center}
\begin{tikzpicture}[scale=0.5]
 \draw[dotted] (0,-.5)--(0,3.5);
 \draw[dotted] (-.5,0)--(9.5,0);
 \draw[thick] (0,0) circle(1.2pt) \U\U\L\L\U\D\L\D\D;
 \draw (2.5,2) node[above] {$2$};
  \draw[red] (6.5,2) node[below] {$L$};
  \draw[thick,red] (6,2)--(7,2);
 \draw (3.5,2) node[above] {$1$};
 \draw (5.6,2.5) node[above] {$1$};
 \draw (6.5,2) node[above] {$2$};
 \draw (7.6,1.5) node[above] {$0$};
 \draw (8.6,0.5) node[above] {$0$};
   \draw[green] (8.3,0.6) node[below] {$D$};
  \draw[thick,green] (8,1)--(9,0);
 \draw (11,1) node{$\mapsto$};
 \draw (11,1) node[above]{$\bijj$};
 \draw[dotted] (13,-.5)--(13,3.5);
 \draw[dotted] (12.5,0)--(21.5,0);
 \draw[thick] (13,0) circle(1.2pt) \U\U\L\L\U\D\D\D;
  \draw (15.5,2) node[above] {$2$};
 \draw (16.5,2) node[above] {$1$};
 \draw (18.6,2.5) node[above] {$1$};
 \draw (19.6,1.5) node[above] {$1$};
  \draw[red] (19.3,1.6) node[below] {$D$};
  \draw[thick,red] (19,2)--(20,1);
  \draw (20.6,0.5) node[above] {$0$};
\end{tikzpicture}
  \end{center}
  \caption{An example of case 2 of the description of the bijection $\bijj:\mathcal{E}_n\to\mathcal{B}_{n-1}$. The labels on the steps indicate the colors.}\label{fig:bijj}
\end{figure}
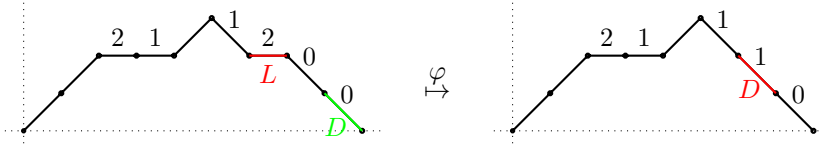

To show that $\bijj$ is a bijection, we describe its inverse. Given a path $M'\in\mathcal{B}_{n-1}$, consider again two cases:
\begin{enumerate}
\item If $M'$ starts with a level step, write $M'=LP$, and let $\bijj^{-1}(M')=UPD$, where the colors of the steps of $P$ are preserved, and the $D$ step at the end gets color $0$.
\item Otherwise, write $M'=P_1D P_2$, where the step between $P_1$ and $P_2$  is the leftmost $D$ step whose color is one less than its height. Call this height $h$, and let $\bijj^{-1}(M')=P_1LP_2D$, where the new $L$ step receives color $h$, the $D$ at the end receives color $0$, and the colors of the steps of $P_1$ and $P_2$ are preserved.\qedhere
\end{enumerate}
\end{proof}

Fig.~\ref{fig:full} gives an example of the full bijection between cyclic permutations with increasing excedances and set partitions. In the encoding of a permutation by a colored Motzkin path via $\bij$, a $D$ step receives color $r$ if the corresponding \close\ in the cycle diagram closes the $r$th open horizontal ray that does not create a cycle with the leftmost vertical ray (which we are forced to close), where available horizontal rays are numbered $0,1,\dots,h-2$ from bottom to top. Similarly, an $L$ step at height $h$ receives color $h$ if the corresponding square in the diagonal sequence is a \ubounce\ (necessarily closing the leftmost open vertical ray), and it receives color $r<h$ 
if the corresponding square is a \lbounce\ closing the $r$th open horizontal ray, numbered increasingly from bottom to top.
The last step in Fig.~\ref{fig:full} illustrates the bijection mentioned in Lemma~\ref{lem:bell}, which is described in~\cite{flajolet_combinatorial_1980}.

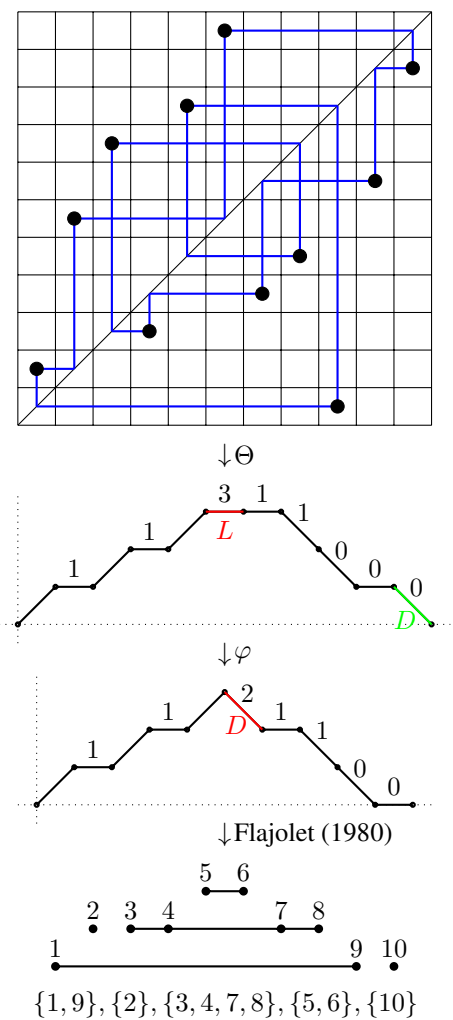
\begin{figure}
\centering
  \begin{tikzpicture}[scale=0.5]
 \draw (0,0) grid (11,11);
 \draw (0,0)--(11,11);
 \newcommand{\pp}{2,6,8,3,9,11,4,5,1,7,10}
    \foreach \y [count=\x] in \pp {
     \draw[blue,thick] (\x-0.5,\x-0.5)--(\x-0.5,\y-0.5)--(\y-0.5,\y-0.5);
     \draw[fill] (\x-0.5,\y-0.5) circle (5pt);
 }
 \draw (5.5,-.8)  node{$\downarrow$};
 \draw (5.5,-.8)  node[right]{$\bij$};

\begin{scope}[shift={(0,-5.3)}]
 \draw[dotted] (0,-.5)--(0,3.5);
 \draw[dotted] (-.5,0)--(11.5,0);
 \draw[thick] (0,0) circle(1.5pt) \U\L\U\L\U\L\L\D\D\L\D;
 \draw (1.5,1) node[above] {$1$};
 \draw (3.5,2) node[above] {$1$};
 \draw (5.5,3) node[above] {$3$};
 \draw (6.5,3) node[above] {$1$};
 \draw (7.6,2.5) node[above] {$1$};
 \draw (8.6,1.5) node[above] {$0$};
 \draw (9.5,1) node[above] {$0$};
 \draw (10.6,0.5) node[above] {$0$};
 \draw[red] (5.5,3) node[below] {$L$};
 \draw[thick,red] (5,3)--(6,3);
 \draw[green] (10.3,0.6) node[below] {$D$};
  \draw[thick,green] (10,1)--(11,0);
  \draw (5.5,-0.8) node{$\downarrow$};
   \draw (5.5,-.8)  node[right]{$\bijj$};
\end{scope}

\begin{scope}[shift={(0.5,-10.1)}]
 \draw[dotted] (0,-.5)--(0,3.5);
 \draw[dotted] (-.5,0)--(10.5,0);
 \draw[thick] (0,0) circle(1.5pt) \U\L\U\L\U\D\L\D\D\L;
 \draw (1.5,1) node[above] {$1$};
 \draw (3.5,2) node[above] {$1$};
 \draw (5.6,2.5) node[above] {$2$};
 \draw (6.5,2) node[above] {$1$};
 \draw (7.6,1.5) node[above] {$1$};
 \draw (8.6,0.5) node[above] {$0$};
 \draw (9.5,0) node[above] {$0$};
 \draw[red] (5.3,2.6) node[below] {$D$};
 \draw[thick,red] (5,3)--(6,2);
 \draw (5,-.8) node{$\downarrow$};
    \draw (5,-.8)  node[right]{\cite{flajolet_combinatorial_1980}};
\end{scope}

\begin{scope}[shift={(0,-14.4)}]
 \draw[thick,fill]  (1,0) circle(2.5pt) node[above] {$1$} --(9,0) circle(2.5pt) node[above] {$9$};
 \draw[thick,fill]  (10,0) circle(2.5pt) node[above] {$10$};
 \draw[thick,fill]  (2,1) circle(2.5pt) node[above] {$2$};
 \draw[thick,fill]  (3,1) circle(2.5pt) node[above] {$3$} --(4,1) circle(2.5pt) node[above] {$4$} -- (7,1) circle(2.5pt) node[above] {$7$} --(8,1) circle(2.5pt) node[above] {$8$};
 \draw[thick,fill]  (5,2) circle(2.5pt) node[above] {$5$} --(6,2) circle(2.5pt) node[above] {$6$};
 \draw (5.5,-1) node{$\{1,9\},\{2\},\{3,4,7,8\},\{5,6\},\{10\}$};
\end{scope}
\end{tikzpicture}
\caption{The bijection between cyclic permutations with increasing excedances and set partitions.}\label{fig:full}
\end{figure}

\ms

An argument similar to the derivation of Equation~\eqref{eq:Iexc} can be used to construct permutations with increasing weak excedances, with the only modification being the disallowance of fixed points at height $h\ge1$. The resulting formula differs from Equation~\eqref{eq:Iexc} only in that now $\ell_h=vw+h$ for $h\ge1$. Setting $x=v=w=t=1$, we obtain precisely the generating function from Lemma~\ref{lem:bell}, implying that the number of permutations in $\S_n$ whose subsequence of weak excedances is increasing is again~$B_n$. 
A direct bijection between such permutations and set partitions can be obtained by declaring $i$ and $\pi(i)$ to be in the same block of the partition for every $i$ with $\pi(i)<i$; equivalently, by erasing from the array of the permutation all the boxes on or above the diagonal, and interpreting the remaining filling of the staircase as a partition, as in~\cite[Fig.~4]{krattenthaler_growth_2006}. For each $k\ge2$, this bijection specializes to a bijection between 
permutations in $\S_n((k{+}1)k\dots1)$ with increasing weak excedances and set partitions of $\{1,2,\dots,n\}$ avoiding {\em $k$-nestings}, as defined in~\cite{burrill_generating_2016}. 

\subsection{Permutations with unimodal cycles and other variations}\label{sec:unimodal}

The conditions in Definition~\ref{def:unimodal} can be considered separately when building cycle diagrams from diagonal sequences. For example, to obtain permutations with unimodal cycles, we simply require that every \close\ completes a cycle by closing two connected open rays. Letting $\UC$ denote the set of permutations with unimodal cycles, we have that 
$\sum_{\pi\in\UC} x^{\fp(\pi)}v^{\exc(\pi)}w^{\dexc(\pi)}t^{\cyc(\pi)}z^{|\pi|}=J_{\d,\l}(z)$
with $d_h=hvt$ and $\ell_h=xt+h(1+vw)$ for all $h$.

It is also possible to get a closed form for the corresponding exponential generating function by using the symbolic method. For $n\ge2$, the generating function for unimodal cycles of size $n$ with respect to the statistics $\fp$, $\exc$ and $\dexc$ is $v(1+vw)^{n-2}$, since each entry other than $1$ and $n$ can be placed in the cycle notation $(1,\dots,n,\dots)$ either after or before $n$, contributing an excedance and a double excedance in the latter case. Thus, the exponential generating function for all unimodal cycles is $$xz+\sum_{n\ge0}v(1+vw)^{n-2}\frac{z^n}{n!}=xz+v\,\frac{e^{(1+vw)z}-1-(1+vw)z}{(1+vw)^2}.$$ Taking sets of such cycles, we get
$$\sum_{\pi\in\UC} x^{\fp(\pi)}v^{\exc(\pi)}w^{\dexc(\pi)}t^{\cyc(\pi)}\frac{z^{|\pi|}}{|\pi|!}=\exp\left(txz+tv\,\frac{e^{(1+vw)z}-1-(1+vw)z}{(1+vw)^2}\right).$$
Setting $x=v=w=t=1$ in the last formula, we obtain sequence~\cite[A187251]{OEIS}.

\ms

Next we consider the set $\UCE$ of permutations with unimodal cycles and increasing excedances. 
Note that there is no immediate way to adapt the symbolic method approach described above to this case. However, our cycle diagram approach is well suited to enumerate these permutations. To obtain the cycle diagram of a permutation in $\UCE$, every \ubounce\ must close the leftmost open vertical ray, and every \close\ must close the leftmost open vertical ray and its matching open horizontal ray, leaving only one possibility in each case. It follows that
$\sum_{\pi\in\UCE} x^{\fp(\pi)}v^{\exc(\pi)}w^{\dexc(\pi)}t^{\cyc(\pi)}z^{|\pi|}=J_{\d,\l}(z),$
where $\ell_0=xt$, $d_h=vt$ and $\ell_h=xt+vw+h$ for all $h\ge1$.

\ms

To obtain permutations with unimodal noncrossing cycles, as in Definition~\ref{def:unimodal}(a)(b), each \ubounce\ must close the rightmost open vertical ray, each \lbounce\ must close the uppermost open horizontal ray, and each \close\ must close the two innermost open rays. It follows that the generating function for these permutations with respect to the usual statistics (including now the number of inversions as well) is
$J_{\d,\l}(z)$ with $\ell_0=xt$, $d_h=vtq^{4h-3}$ and $\ell_h=xtq^{2h}+(1+vw)q^{2h-1}$ for $h\ge1$.

Setting $x=v=w=t=q=1$, we obtain the continued fraction $F(z):=J_{\d,\l}(z)$ with $\ell_0=1$, $d_h=1$ and $\ell_h=3$ for $h\ge1$, which can be written as $F(z)=1/(1-z-B(z))$, where $B(z)=z^2/(1-3z-B(z))$. We can solve this equation to obtain the closed form $$F(z)=\frac{2}{1+z+\sqrt{1-6z+5z^2}}.$$ The coefficients give the sequence~\cite[A033321]{OEIS}, which also counts permutations avoiding certain triples of patterns of length $4$.

The same sequence is obtained when counting permutations with increasing excedances and increasing deficiencies. Indeed, to build the cycle diagrams of such permutations, every \ubounce, \lbounce\ and \close\ is forced to close the outermost open rays. Even though we are unable to keep track of the number of cycles (as happened in the case of $321$-avoiding permutations from Section~\ref{sec:321}), the generating function for permutations with increasing excedances and increasing deficiencies with respect to the other usual statistics is $J_{\d,\l}(z)$ with $\ell_0=x$, $d_h=vq^{2h-1}$ and $\ell_h=x+(1+vw)q^{h}$ for $h\ge1$.

\subsection{Involutions}
Most of our results can be easily adapted to involutions, that is, permutations equal to their inverse.
The cycle diagrams of involutions are those that are symmetric along the main diagonal. In particular, they do not have any diagonal squares of type \ubounce\ or \lbounce, and each \close\ must close an open vertical ray and its symmetric open horizontal ray.

For example, the continued fraction for involutions with respect to our usual statisics (note that involutions have no double excedances) is
$J_{\d,\l}(z)$ with $d_h=vtq^{2h-1}[h]_{q^2}$ and $\ell_h=xtq^{2h}$ for all $h$.
If we require the involutions to be $321$-avoiding, then we get $J_{\d,\l}(z)$ with $\ell_0=xt$, $d_h=vtq^{2h-1}$ and $\ell_h=0$ for $h\ge1$.

\section{Pattern-avoiding cyclic permutations}\label{sec:pattern}

In this section we discuss a related research direction, namely the problem of enumerating pattern-avoiding cyclic permutations, and describe what is known in this area.

\subsection{Classical patterns}

Given a pattern $\sigma$, let $\C_n(\sigma)=\C_n\cap\S_n(\sigma)$ be the set of cyclic permutations that avoid $\sigma$.
Similarly, for a set of patterns $\Sigma$, denote by $\C_n(\Sigma)$ the set of cyclic permutations avoiding all the patterns in~$\Sigma$.

The following question was posed by Richard Stanley at the {\it Permutation Patterns 2010} conference held at Dartmouth College.

\begin{question}\label{q:stanley}
For given $\sigma\in\S_k$, find a formula for $|\C_n(\sigma)|$.
\end{question}

The present paper was originally motivated by this question, which remains open for all patterns of length $|\sigma|\ge3$. Focusing on $\sigma=321$, recall from Lemma~\ref{lem:321} that $321$-avoiding permutations are those whose excedances and non-excedances are increasing. In Theorem~\ref{thm:bell} we have enumerated cyclic permutations that satisfy the first condition, but there is no obvious way to incorporate the second one.

Related to Question~\ref{q:stanley}, one may consider cyclic permutations that avoid multiple patterns. For some very specific sets of patterns, the enumeration of cyclic permutations avoiding them was done in~\cite{archer_cyclic_2014}. In the following theorem, $\mu$ denotes the M\"obius function.

\begin{theorem}[\cite{archer_cyclic_2014}]  
For $n\ge2$,
\begin{align*}
&|\C_n(213,312)|=|\C_n(132,231)|=\frac{1}{2n}\sum_{d|n \atop \text{$d$ odd}} \mu(d) 2^{n/d},\\
&|\C_n(321,2143,3142)|=\frac{1}{n}\sum_{d|n} \mu(d) 2^{n/d},\\
&|\C_n(123,2413,3412)|=\begin{cases} \frac{1}{n}\sum_{d|n} \mu(d) 2^{n/d} & \text{if $n\not\equiv2\bmod4$,} \\ 
\frac{1}{n}\sum_{d|n} \mu(d) 2^{n/d}+\frac{2}{n}\sum_{d|\frac{n}{2}} \mu(d) 2^{n/2d} & \text{if $n\equiv2\bmod 4$.} \end{cases}
\end{align*}
\end{theorem}
 
\subsection{Consecutive patterns} 

\cite{gessel_counting_1993} expressed the number of permutations with a given cycle structure and given descent set (equivalently, a given set of positions of occurrences of $\ul{21}$) as a scalar product of symmetric functions.

In~\cite[Thm.\ 6.1]{gessel_counting_1993}, they give a generating function for cyclic permutations according to the number of descents, proved using quasisymmetric functions. 
The distribution of ascent sets (equivalently, occurrences of $\ul{12}$) on $\C_n$ agrees with that of descent sets as long as $n\not\equiv2\bmod 4$, as shown algebraically in~\cite[Thm.\ 4.1]{gessel_counting_1993}, and combinatorially in~\cite[Cor.\ 3.1]{steinhardt_permutations_2010} and~\cite[Prop.\ 3.13]{archer_cyclic_2014}. 
In addition, recursive formulas for the number of cycles in $\C_n$ with $k$ descents (resp. with $k$ ascents), including the case $n\equiv2\bmod 4$, are given and proved combinatorially in~\cite{archer_cyclic_2014}. 
 
Regarding the distribution of occurrences of longer consecutive patterns in cyclic permutations, formulas enumerating $\ul{123}$-avoiding (and $\ul{321}$-avoiding) cycles have been recently obtained in~\cite{elizalde_exact_2017}.

\section{From combinatorial sequences to weight sequences}\label{sec:digression}

In this section we regard the above method for translating between permutations and colored Motzkin paths in a more abstract setting.

Given two sequences $\d=(d_1,d_2,\dots)$ and $\l=(\ell_0,\ell_1,\ell_2,\dots)$, one can define a sequence $\a=(a_0,a_1,a_2,\dots)$ by
\begin{equation}\label{eq:adl}\sum_{n\ge0} a_n z^n= \dfrac{1}{1-\ell_0z-\dfrac{d_1z^2}{1-\ell_1z-\dfrac{d_2z^2}{1-\ell_2z-\dfrac{d_3z^2}{\ddots}}}}.
\end{equation}
As we have used repeatedly in the paper, the coefficient $a_n$ counts weighted Motzkin paths of length $n$ where $D$ steps (resp. $L$ steps) at height $h$ have weight $d_h$ (resp. $\ell_h$) for each $h\ge0$. In particular, if  $d_h,\ell_h\in\Z_{\ge0}$ for all $h$, then  $a_n\in\Z_{\ge0}$ for all $n$. Note also that $a_0=1$, corresponding to the empty path.

It is interesting to consider the inverse construction. Given a sequence $\a=(a_0,a_1,a_2,\dots)$ with $a_0=1$, one can sometimes solve Equation~\eqref{eq:adl} for $\d$ and $\l$. When the solution exists, it is unique. To see this, expand the right-hand side of~\eqref{eq:adl} to get
$$1+\ell_0z+(\ell_0^2+d_1)z^2+(\ell_0^3+2\ell_0d_1+\ell_1d_1)z^3+(\ell_0^4+3l_0^2d_1+2l_0l_1d_1+l_1^2d_1+d_1^2+d_1d_2)z^4+\cdots.$$
In general, equating coefficients of $z^n$ in Equation~\eqref{eq:adl} for even $n$, say $n=2m$, we get
$$a_{2m}=d_1d_2\dots d_m+\text{polynomial}(\ell_0,\ell_1,\dots,\ell_{m-1},d_1,d_2,\dots,d_{m-1}),$$
whereas for odd $n$, say $n=2m+1$, we get
$$a_{2m+1}=d_1d_2\dots d_m\ell_m+\text{polynomial}(\ell_0,\ell_1,\dots,\ell_{m-1},d_1,d_2,\dots,d_{m}),$$
as determined by the heights of $L$ and $D$ steps that Motzkin paths of length $n$ can have.
(The notation $\text{polynomial}(x_1,x_2,\dots)$ stands for some polynomial in the variables $x_1,x_2,\dots$.)
Thus, as long as $d_i\neq0$ for all $i$, one can solve for $\d$ and $\l$ to obtain
\begin{align*} d_m&=\frac{a_{2m}-\text{polynomial}(\ell_0,\ell_1,\dots,\ell_{m-1},d_1,d_2,\dots,d_{m-1})}{d_1d_2\dots d_{m-1}},\\
\ell_m&=\frac{a_{2m+1}-\text{polynomial}(\ell_0,\ell_1,\dots,\ell_{m-1},d_1,d_2,\dots,d_{m})}{d_1d_2\dots d_m}.
\end{align*}
The first few terms of these sequences are
$$\ell_0=a_1, \quad  d_1=a_2-\ell_0^2,\quad \ell_1=\frac{a_3-\ell_0^3-2\ell_0d_1}{d_1}, \quad d_2=\frac{a_4-\ell_0^4-3l_0^2d_1-2l_0l_1d_1-l_1^2d_1-d_1^2}{d_1},\quad\dots$$
In general, the relation between the sequence $\a$ and the sequences $\d$ and $\l$ is given by Stieltjes's expansion theorem for J-fractions~(see \cite{stieltjes_sur_1889}, \cite{wall_analytic_1948} and \cite[Thm.\ S]{flajolet_combinatorial_1980}) in terms of the so-called {\em Stieltjes matrix}.

For an arbitrary sequence $\a$ of nonnegative integers with $a_0=1$, the corresponding sequences $\d$ and $\l$ given by Equation~\eqref{eq:adl} 
may not exist, and even when they do, they typically do not consist of nonnegative integers. However, some experimentation shows that many combinatorial sequences $\a$ seem to correspond to sequences $\d,\l$ of nonnegative integers, as shown in Tab.~\ref{tab:adl}. This means that such sequences can be interpreted as counting weighted Motzkin paths, often with a simple weight function. This raises the questions of which combinatorial sequences have this property, which ones give positive Motzkin weights, and how the nature of the generating function for $\a$ (e.g.\ algebraic or D-finite) is related to the behavior of $\d$ and $\l$. An example of a sequence that does not give nice weights is the one for Baxter numbers~\cite[A001181]{OEIS}, where the corresponding sequences $\d$ and $\l$ contain fractional negative entries. Another example is the counting sequence of $1342$-avoiding permutations, for which the weights are not integers but they are positive and have a simple expression.

\begin{table}[htb]
$$
\begin{array}{|c|c|c||c|c|}
  \hline
  \text{Name of sequence} & \text{OEIS \cite{OEIS}} & a_n & d_h & \ell_h \\
  \hline\hline
  \text{Catalan} & \text{A000108} & C_n & 1 & \begin{cases} 1, & h=0 \\ 2, & h\ge1 \end{cases} \\
  \hline
  \text{Motzkin} & \text{A001006} & M_n & 1 & 1 \\
  \hline
  \text{central binomial} & \text{A000984} & \displaystyle\binom{2n}{n} & \begin{cases} 2, & h=1 \\ 1, & h\ge2 \end{cases} & 2 \\
  \hline
  \text{central trinomial} & \text{A002426} & [z^n]\frac{1}{\sqrt{1-2z-3z^2}} & \begin{cases} 2, & h=1 \\ 1, & h\ge2 \end{cases} & 1 \\
  \hline
  \text{(large) Schr\"oder} & \text{A006318} & S_n & 2 & \begin{cases} 2, & h=0 \\ 3, & h\ge1 \end{cases} \\
  \hline
  \text{Bell} & \text{A000110} & B_n & h & h+1 \\
  \hline
  \text{\bt{c}set partitions \\ with no singletons \et} & \text{A000296} & n!\,[z^n]e^{e^z-1-z} & h & h \\
  \hline
  \text{factorial} & \text{A000142}& n! & h^2 & 2h+1 \\
  \hline
  \text{odd double factorial} & \text{A001147}& 1\cdot3\cdots(2n-1) & 2h(2h-1) & 4h+1 \\
  \hline
  \text{even double factorial} & \text{A000165}& 2\cdot4\cdots(2n) & 4h^2 & 4h+2 \\
  \hline
  \text{derangements} & \text{A000166} & n! \sum_{i=0}^n \frac{(-1)^{i}}{i!} & h^2 & 2h \\
  \hline
  \text{Euler} & \text{A000111} & E_n & \displaystyle\binom{h+1}{2} & h+1 \\
  \hline
  \text{$\ul{123}$-avoiding} & \text{A049774}& |\S_n(\ul{123})| & h^2 & h+1 \\
   \hline
  \text{labeled graphs} & \text{A006125} & 2^{\binom{n}{2}} & 8^{h-1}(2^h-1) & 2^{h-1}(3\cdot2^h-1) \\
  \hline
  \text{unsigned even Genocchi} & \text{A110501} & G_{2n} & h^3(h+1)  & (h+1)(2h+1) \\
  \hline
  \text{median Genocchi} & \text{A005439} & H_{2n+1} & h^4  & 2h(h+1)+1 \\
  \hline
  \end{array}
$$
\caption{Examples of combinatorial sequences and their corresponding Motzkin weights.}
\label{tab:adl}
\end{table}

This correspondence provides an alternative approach to finding (or conjecturing) an expression for a sequence $\a$ for which only the first few terms are known: one can compute the corresponding sequences of weights $\d$ and $\l$, and check if they appear to have a simple formula. As shown in Tab.~\ref{tab:adl}, this method may work even when the generating function for $\a$ is not D-finite.

\acknowledgements
The author thanks Jiang Zeng for providing useful references. 

\bibliographystyle{abbrvnat}
\bibliography{cyc_avoid_bib}

\end{document}